\RequirePackage{fix-cm}
\documentclass[smallextended]{svjour3}       % onecolumn (second format)
\smartqed  % flush right qed marks, e.g. at end of proof
\usepackage{graphicx}
\usepackage{amsmath}
\usepackage{amsfonts}
\usepackage{amssymb}
\begin{document}

\title{Approximation by sequence of operators including Dunkl Appell polynomials}

\author{Sezgin Sucu}

\institute{S. Sucu \at
              Ankara University Faculty of Science \\
              Department of Mathematics, TR-06100, Ankara, Turkey.\\
              \email{ssucu@ankara.edu.tr}
}

\date{Received: date / Accepted: date}

\maketitle

\begin{abstract}
In this article, we give a sequence of operators for producing an
approximation result. The relation between the rate of approximation of
sequence operators including Dunkl variant of exponential function with
first and second-order modulus of continuity are shown. A specific
application of sequence of operators which include Gould-Hopper type
polynomials is constructed.
\end{abstract}

\section{Introduction}
\label{intro}

In applied science and engineering, the approximation theory plays a
significant role. The main purpose of the theory of approximation is related
to representation of functions which are difficult to comprehend of its
characteristic properties. In this area, the first result that a continuous
function which is defined on a closed and bounded interval can be
represented with the help of polynomials was expressed and proved by
Weierstrass in 1885. With this theorem proved by Weierstrass, there have
been important developments in other fields of theoretical mathematics and
applied mathematics.

The different proofs of Weierstrass approximation theorem were given by
Weierstrass, Picard, Lebesgue, Fej\'{e}r, Mittag-Leffler, Landau, de la Vall%
\'{e}e-Poussin, Bernstein and Montel using dissimilar methods. By means of
probabilistic considerations, Bernstein gave an elementary proof of this
theorem by giving an explicit construction of the polynomials. For each
function $\varphi $ which is continuous on $\left[ 0,1\right] $, the
corresponding polynomials having degree $n$ are defined%
\begin{equation*}
B_{n}\left( \varphi ;x\right) =\sum\limits_{i=0}^{n}\binom{n}{i}x^{i}\left(
1-x\right) ^{n-i}\varphi \left( \frac{i}{n}\right) \text{,\ }x\in \left[ 0,1%
\right]
\end{equation*}%
by Bernstein in 1912 \cite{1}. In 1950, the following operators%
\begin{equation}
M_{n}\left( \varphi ;x\right) =e^{-nx}\sum\limits_{i=0}^{\infty }\frac{%
\left( nx\right) ^{i}}{i!}\varphi \left( \frac{i}{n}\right)  \label{1}
\end{equation}%
was carefully examined by Szasz \cite{2} who proved that the sequence $%
\left\{ M_{n}\left( \varphi ;x\right) \right\} _{n\geq 1}$ converges to $%
\varphi \left( x\right) $ for all $x\in \left[ 0,\infty \right) $ provided
the function $\varphi $ is continuous.

Let $r$ be holomorphic in $\left\{ \omega :\left\vert \omega \right\vert <R%
\text{, }R>1\right\} $ such that $r\left( 1\right) \neq 0$ and suppose that
it has a Taylor expansion as the following form $r\left( \omega \right)
=\sum\limits_{i=0}^{\infty }a_{i}\omega ^{i}$ $\left( a_{0}\neq 0\right) $.
The Appell polynomials $q_{i}$ are defined as%
\begin{equation}
r\left( \omega \right) e^{\omega x}=\sum\limits_{i=0}^{\infty }q_{i}\left(
x\right) \omega ^{i}\text{.}  \label{2}
\end{equation}%
Following operators $J_{n}\left( \varphi ;x\right) $ with $q_{i}\left(
x\right) \geq 0$ for $x\in \left[ 0,\infty \right) $,%
\begin{equation}
J_{n}\left( \varphi ;x\right) =\frac{e^{-nx}}{r\left( 1\right) }%
\sum\limits_{i=0}^{\infty }q_{i}\left( nx\right) \varphi \left( \frac{i}{n}%
\right)  \label{3}
\end{equation}%
considered in \cite{3}.

Recently, there have been significant developments in theoretical
mathematics and applied mathematics in quantum calculus. In some
applications, $q$ -Bernstein polynomials are more useful than Bernstein
polynomials. Because of the importance of $q$ -Bernstein polynomials in many
fields of application such as pattern recognition and computer aided
geometry design, these kinds of polynomials have been extensively studied in
the literature. There are many papers and several books published on $q$
-Bernstein polynomials (\cite{4},\cite{5},\cite{6},\cite{7},\cite{8}).

For $\mu +1/2>0$ extended exponential function \cite{9} defined by%
\begin{equation}
e_{\mu }\left( x\right) =\sum\limits_{i=0}^{\infty }\frac{x^{i}}{\gamma
_{\mu }\left( i\right) }\text{,}  \label{4}
\end{equation}%
where
\begin{equation}
\gamma _{\mu }\left( 2i\right) =\frac{2^{2i}i!\Gamma \left( i+\mu
+1/2\right) }{\Gamma \left( \mu +1/2\right) }\text{ \ \ and \  }\gamma _{\mu
}\left( 2i+1\right) =\frac{2^{2i+1}i!\Gamma \left( i+\mu +3/2\right) }{%
\Gamma \left( \mu +1/2\right) }  \label{5}
\end{equation}%
has been used to solve the approximation problem of the sequence of
operators defined as \cite{10}%
\begin{equation}
S_{n}\left( \varphi ;x\right) =\frac{1}{e_{\mu }\left( nx\right) }%
\sum\limits_{i=0}^{\infty }\frac{\left( nx\right) ^{i}}{\gamma _{\mu
}\left( i\right) }\varphi \left( \frac{i+2\mu \theta _{i}}{n}\right) \text{.}
\label{6}
\end{equation}%
If $\theta _{i}$ is defined as%
\begin{equation*}
\theta _{i}=\left\{
\begin{array}{ccc}
0 & ; & i\in \left\{ 0,2,4,...,2n,...\right\}  \\
1 & ; & i\in \left\{ 1,3,...,2n+1,...\right\}
\end{array}%
\right.
\end{equation*}%
then following recursion relation%
\begin{equation}
\gamma _{\mu }\left( i+1\right) =\left( i+1+2\mu \theta _{i+1}\right) \gamma
_{\mu }\left( i\right) \text{,\ }i\in
%TCIMACRO{\U{2115} }%
%BeginExpansion
\mathbb{N}
%EndExpansion
_{0}  \label{7}
\end{equation}%
is satisfied.

There are many authors considering the exponential function based on gamma
function. For the approximation results related to $\left( \ref{4}\right) $,
we refer the reader to studies by authors (\cite{11},\cite{12},\cite{13},%
\cite{14}) who proved the quantitative convergence theorems for the
different kinds of operators sequence generated by the exponential function
based on gamma function.

The Dunkl operator $\Lambda _{\mu }$ has the form
\begin{equation}
\left( \Lambda _{\mu }\vartheta \right) \left( x\right) =\Lambda _{\mu
}\vartheta \left( x\right) =\frac{d\vartheta }{dx}+\mu \frac{\vartheta
\left( x\right) -\vartheta \left( -x\right) }{x}\text{,}  \label{8}
\end{equation}%
where $\mu $ is a real number satisfying $\mu >-1/2$ and $\vartheta $ is an
entire function \cite{9}. Repeating this process we can define
\begin{equation}
\left( \Lambda _{\mu }^{2}\vartheta \right) \left( x\right) =\Lambda _{\mu
}^{2}\vartheta \left( x\right) =\frac{d^{2}\vartheta }{dx^{2}}+\frac{2\mu }{x%
}\frac{d\vartheta }{dx}-\mu \frac{\vartheta \left( x\right) -\vartheta
\left( -x\right) }{x^{2}}\text{.}  \label{9}
\end{equation}

The starting point of our work here is the paper of Ben Cheikh and Gaied
\cite{15}. A polynomial set $\left( q_{i}\right) $ is called Dunkl-Appell
polynomial set if and only if for $i\in
%TCIMACRO{\U{2115} }%
%BeginExpansion
\mathbb{N}
%EndExpansion
_{0}$%
\begin{equation*}
\Lambda _{\mu }q_{i+1}\left( x\right) =\frac{\gamma _{\mu }\left( i+1\right)
}{\gamma _{\mu }\left( i\right) }q_{i}\left( x\right) \text{.}
\end{equation*}%
The authors have proved that the following propositions are equivalent:

$\left( i\right) $ $\left( q_{i}\right) $ is a Dunkl-Appell polynomial set.

$\left( ii\right) $ The polynomials $q_{i}$ can be written as%
\begin{equation*}
q_{i}\left( x\right) =\sum_{j=0}^{i}\binom{i}{j}_{\mu }a_{i-j}x^{j}\text{, }%
\left( a_{0}\neq 0\right) \text{,}
\end{equation*}%
where the sequence $\left( a_{j}\right) $ is not depend on $i$ and $\binom{i%
}{j}_{\mu }=\frac{\gamma _{\mu }\left( i\right) }{\gamma _{\mu }\left(
j\right) \gamma _{\mu }\left( i-j\right) }$.

$\left( iii\right) $ $\left( q_{i}\right) $ is generated by%
\begin{equation*}
Q\left( t\right) e_{\mu }\left( xt\right) =\sum_{i=0}^{\infty }\frac{%
q_{i}\left( x\right) }{\gamma _{\mu }\left( i\right) }t^{i}
\end{equation*}%
where%
\begin{equation*}
Q\left( t\right) =\sum_{i=0}^{\infty }\frac{a_{i}}{\gamma _{\mu }\left(
i\right) }t^{i}\text{, }\left( a_{0}\neq 0\right) \text{.}
\end{equation*}%
The main purpose of this article is to explain the approximation problem
described below in detail:

\textbf{(P) }Find the sequence of operators generated by Dunkl-Appell
polynomial set $\left( q_{i}\right) $ defined by \cite{15}%
\begin{equation}
Q\left( t\right) e_{\mu }\left( xt\right) =\sum_{i=0}^{\infty }\frac{%
q_{i}\left( x\right) }{\gamma _{\mu }\left( i\right) }t^{i}\text{,}
\label{10}
\end{equation}%
where the coefficients $\gamma _{\mu }$ are defined as in $\left( \ref{5}%
\right) $.

For $\mu \geq 0$, the sequence of operators providing a solution of the
above-mentioned problem \textbf{(P) }is expressed as follows:

\begin{equation}
K_{n}^{\mu }\left( f;x\right) =\frac{1}{Q\left( 1\right) e_{\mu }\left(
nx\right) }\sum_{i=0}^{\infty }\frac{q_{i}\left( nx\right) }{\gamma _{\mu
}\left( i\right) }f\left( \frac{i+2\mu \theta _{i}}{n}\right) \text{.}
\label{11}
\end{equation}%
The operators $\left( \ref{11}\right) $ are the generalized form of
operators $\left( \ref{3}\right) $. When we take $\mu =0$ in $\left( \ref{11}%
\right) $, mentioned expression is obtained explicitly.

Following section gives a sequence of operators for producing an
approximation. The relation between the approximation order of operators
defined by $\left( \ref{11}\right) $ with first and second-order modulus of
continuity are shown. A specific application of sequence of operators given
in $\left( \ref{11}\right) $ which include Gould-Hopper type polynomials is
constructed.

\section{Main Results}
\label{sec:1}

The starting point for this construction is the some crucial preliminary
results that are used in this article.

$\left( \ref{11}\right) $ yields the following assertion:

\begin{lemma}
$K_{n}^{\mu }$ operators are linear and positive operators.
\end{lemma}

\begin{lemma}
Let $\left\{ K_{n}^{\mu }\right\} _{n\geq 1}$ be the sequence of operators
defined by $\left( \ref{11}\right) $. Then there are following assertions:%
\begin{eqnarray}
K_{n}^{\mu }\left( 1;x\right) &=&1\text{,}  \notag \\
K_{n}^{\mu }\left( \xi ;x\right) &=&x+\frac{\left( e_{\mu }\left( nx\right)
-e_{\mu }\left( -nx\right) \right) Q^{^{\prime }}\left( 1\right) +e_{\mu
}\left( -nx\right) \left( \Lambda _{\mu }Q\right) \left( 1\right) }{Q\left(
1\right) ne_{\mu }\left( nx\right) }\text{,}  \notag \\
K_{n}^{\mu }\left( \xi ^{2};x\right) &=&x^{2}+\frac{e_{\mu }\left( nx\right)
\left( 2Q^{^{\prime }}\left( 1\right) +Q\left( 1\right) \right) +2\mu
Q\left( -1\right) e_{\mu }\left( -nx\right) }{Q\left( 1\right) ne_{\mu
}\left( nx\right) }x  \notag \\
&&+\frac{\left( \Lambda _{\mu }Q\right) \left( 1\right) e_{\mu }\left(
-nx\right) }{Q\left( 1\right) n^{2}e_{\mu }\left( nx\right) }  \notag \\
&&+\frac{\left[ 2Q^{^{^{\prime \prime }}}\left( 1\right) -\left( \Lambda
_{\mu }Q\right) ^{^{\prime }}\left( 1\right) -\left( \Lambda _{\mu
}Q^{^{\prime }}\right) \left( 1\right) +Q^{^{\prime }}\left( 1\right) -2\mu
Q^{^{\prime }}\left( -1\right) \right] \left( e_{\mu }\left( nx\right)
-e_{\mu }\left( -nx\right) \right) }{Q\left( 1\right) n^{2}e_{\mu }\left(
nx\right) }  \notag \\
&&+\frac{\left( \Lambda _{\mu }^{2}Q\right) \left( 1\right) +2\mu \left(
\Lambda _{\mu }Q\right) \left( -1\right) }{Q\left( 1\right) n^{2}}\text{.}
\label{12}
\end{eqnarray}
\end{lemma}

\begin{proof}
The proof of first relation$\ $in $\left( \ref{12}\right) $ is quite simple,
because it is enough to replace $x$ with $nx$ and $t$ with $1$ in the
following statement%
\begin{equation*}
\sum_{i=0}^{\infty }\frac{q_{i}\left( x\right) }{\gamma _{\mu }\left(
i\right) }t^{i}=Q\left( t\right) e_{\mu }\left( xt\right) \text{.}
\end{equation*}

To obtain the required second statement in $\left( \ref{12}\right) $ it
suffices to apply the Dunkl operator $\Lambda _{\mu }$ to the both side of
above equality. By reason of the fact that
\begin{equation}
\Lambda _{\mu }\left( t^{i}\right) :=\Lambda _{\mu ,t}\left( t^{i}\right) =%
\frac{\gamma _{\mu }\left( i\right) }{\gamma _{\mu }\left( i-1\right) }%
t^{i-1}\text{ \ and \ }\Lambda _{\mu }\left( e_{\mu }\left( xt\right)
\right) :=\Lambda _{\mu ,t}\left( e_{\mu }\left( xt\right) \right) =xe_{\mu
}\left( xt\right)  \label{13}
\end{equation}%
and the following product rule
\begin{equation}
\left( \Lambda _{\mu }\left( \vartheta \zeta \right) \right) \left( t\right)
=\vartheta \left( t\right) \Lambda _{\mu }\zeta \left( t\right) +\zeta
\left( -t\right) \Lambda _{\mu }\vartheta \left( t\right) +\vartheta
^{^{\prime }}\left( t\right) \left( \zeta \left( t\right) -\zeta \left(
-t\right) \right)  \label{14}
\end{equation}%
it is clear that%
\begin{equation*}
\sum_{i=1}^{\infty }\frac{q_{i}\left( x\right) }{\gamma _{\mu }\left(
i\right) }\left( i+2\mu \theta _{i}\right) t^{i-1}=xQ\left( t\right) e_{\mu
}\left( xt\right) +e_{\mu }\left( -xt\right) \left( \Lambda _{\mu }Q\right)
\left( t\right) +Q^{^{\prime }}\left( t\right) \left[ e_{\mu }\left(
xt\right) -e_{\mu }\left( -xt\right) \right] \text{.}
\end{equation*}%
To achieve the third statement in $\left( \ref{12}\right) $ it is enough
again to apply the Dunkl operator $\Lambda _{\mu }$ to the both side of
above equality. Similarly, by virtue of $\left( \ref{13}\right) $ and $%
\left( \ref{14}\right) $, we conclude that%
\begin{eqnarray*}
\sum_{i=2}^{\infty }\frac{q_{i}\left( x\right) }{\gamma _{\mu }\left(
i\right) }\left( i+2\mu \theta _{i}\right) \left( i-1+2\mu \theta
_{i-1}\right) t^{i-2} &=&x^{2}Q\left( t\right) e_{\mu }\left( xt\right)
+xe_{\mu }\left( -xt\right) \left( \Lambda _{\mu }Q\right) \left( t\right) \\
&&+xQ^{^{\prime }}\left( t\right) \left[ e_{\mu }\left( xt\right) -e_{\mu
}\left( -xt\right) \right] \\
&&-xe_{\mu }\left( -xt\right) \left( \Lambda _{\mu }Q\right) \left( t\right)
+e_{\mu }\left( xt\right) \left( \Lambda _{\mu }^{2}Q\right) \left( t\right)
\\
&&+\left[ e_{\mu }\left( -xt\right) -e_{\mu }\left( xt\right) \right] \left(
\Lambda _{\mu }Q\right) ^{^{\prime }}\left( t\right) \\
&&+\left[ xe_{\mu }\left( xt\right) +xe_{\mu }\left( -xt\right) \right]
Q^{^{\prime }}\left( t\right) \\
&&+\left[ e_{\mu }\left( -xt\right) -e_{\mu }\left( xt\right) \right] \left(
\Lambda _{\mu }Q^{^{\prime }}\right) \left( t\right) \\
&&+2\left[ e_{\mu }\left( xt\right) -e_{\mu }\left( -xt\right) \right]
Q^{^{^{\prime \prime }}}\left( t\right) \text{.}
\end{eqnarray*}%
If we replace $x$ with $nx$ and $t$ with $1$ in the above expressions, these
equalities enable us to obtain the proof of lemma.
\end{proof}

\begin{theorem}
Let $f$ be a real valued continuous function on $\left[ 0,\infty \right) $
with%
\begin{equation*}
\frac{f\left( x\right) }{1+x^{2}}\rightarrow A\in
%TCIMACRO{\U{211d} }%
%BeginExpansion
\mathbb{R}
%EndExpansion
\text{, }\left( x\rightarrow \infty \right) \text{.}
\end{equation*}%
Then the sequence $\left\{ K_{n}^{\mu }\left( f;.\right) \right\} _{n\geq 1}$
does converge uniformly to the $f$ function on $\left[ a,b\right] $ for
every $0\leq a<b<\infty $.
\end{theorem}

\begin{proof}
The proof of this fact is the substance of Korovkin-type property which is
given by Altomare \cite{16}. The formula obtained in Lemma 2 allows us to get%
\begin{equation*}
K_{n}^{\mu }\left( t^{i};x\right) \rightrightarrows x^{i}\text{, }i=0,1,2,
\end{equation*}%
on every interval $\left[ a,b\right] \subset \left[ 0,\infty \right) $.
Returning now to the universal Korovkin-type property, the theorem is
herewith proved.
\end{proof}

We shall now give a lemma which will be useful later. By combining relations
in $\left( \ref{12}\right) $, we arrive at the following assertion:

For every natural number $n$, we define two notations for convenience as
follows:%
\begin{equation*}
\Omega _{1}:=K_{n}^{\mu }\left( \xi -x;x\right) \text{ \ \ \ and \ \ \ }%
\Omega _{2}:=K_{n}^{\mu }\left( \left( \xi -x\right) ^{2};x\right) \text{.}
\end{equation*}

\begin{lemma}
For $K_{n}^{\mu }$ operators, the following relations%
\begin{eqnarray*}
\Omega _{1} &=&\frac{\left( e_{\mu }\left( nx\right) -e_{\mu }\left(
-nx\right) \right) Q^{^{\prime }}\left( 1\right) +e_{\mu }\left( -nx\right)
\left( \Lambda _{\mu }Q\right) \left( 1\right) }{Q\left( 1\right) ne_{\mu
}\left( nx\right) }\text{,} \\
\Omega _{2} &=&\left( 1+\frac{2e_{\mu }\left( -nx\right) \left[ \mu Q\left(
-1\right) +Q^{^{\prime }}\left( 1\right) -\left( \Lambda _{\mu }Q\right)
\left( 1\right) \right] }{Q\left( 1\right) e_{\mu }\left( nx\right) }\right)
\frac{x}{n} \\
&&+\frac{e_{\mu }\left( -nx\right) \left( \Lambda _{\mu }Q\right) \left(
1\right) }{Q\left( 1\right) n^{2}e_{\mu }\left( nx\right) } \\
&&+\frac{\left[ 2Q^{^{^{\prime \prime }}}\left( 1\right) -\left( \Lambda
_{\mu }Q\right) ^{^{\prime }}\left( 1\right) -\left( \Lambda _{\mu
}Q^{^{\prime }}\right) \left( 1\right) +Q^{^{\prime }}\left( 1\right) -2\mu
Q^{^{\prime }}\left( -1\right) \right] \left( e_{\mu }\left( nx\right)
-e_{\mu }\left( -nx\right) \right) }{Q\left( 1\right) n^{2}e_{\mu }\left(
nx\right) } \\
&&+\frac{\left( \Lambda _{\mu }^{2}Q\right) \left( 1\right) +2\mu \left(
\Lambda _{\mu }Q\right) \left( -1\right) }{Q\left( 1\right) n^{2}}
\end{eqnarray*}%
can be expressed.
\end{lemma}

Let $\delta \geq 0$ and $\mathcal{UC}\left[ 0,\infty \right) $ be the family
of uniformly continuous functions on positive real axis. Assume that $%
\varphi \in \mathcal{UC}\left[ 0,\infty \right) $. For $\zeta _{1},\zeta
_{2}\in \left[ 0,\infty \right) $ such that $\left\vert \zeta _{1}-\zeta
_{2}\right\vert \leq \delta $, the least upper bound of $\left\vert \varphi
\left( \zeta _{1}\right) -\varphi \left( \zeta _{2}\right) \right\vert $ is
represented by $w\left( \varphi ;\delta \right) $ which is called modulus of
continuity of $\varphi $.

We continue this section with the following useful relation between sequence
of operators $\left\{ K_{n}^{\mu }\left( f;.\right) \right\} _{n\geq 1}$
with modulus of continuity.

\begin{theorem}
Let $f$ be a uniformly continuous real valued function on $\left[ 0,\infty
\right) $. Then the following holds
\begin{equation*}
\left\vert K_{n}^{\mu }\left( f;x\right) -f\left( x\right) \right\vert \leq
\left( 1+\lambda _{n}\left( x\right) \right) w\left( f;\frac{1}{\sqrt{n}}%
\right) \text{,}
\end{equation*}%
where $\lambda _{n}\left( x\right) =\sqrt{n\Omega _{2}}$.
\end{theorem}

\begin{proof}
According to Lemma 2 and well known Schwarz inequality, it follows that%
\begin{eqnarray*}
\left\vert K_{n}^{\mu }\left( f;x\right) -f\left( x\right) \right\vert &\leq
&\frac{1}{Q\left( 1\right) e_{\mu }\left( nx\right) }\sum_{i=0}^{\infty }%
\frac{q_{i}\left( nx\right) }{\gamma _{\mu }\left( i\right) }\left\vert
f\left( \frac{i+2\mu \theta _{i}}{n}\right) -f\left( x\right) \right\vert \\
&\leq &\left\{ 1+\frac{1}{\delta }\frac{1}{Q\left( 1\right) e_{\mu }\left(
nx\right) }\sum_{i=0}^{\infty }\frac{q_{i}\left( nx\right) }{\gamma _{\mu
}\left( i\right) }\left\vert \frac{i+2\mu \theta _{i}}{n}-x\right\vert
\right\} w\left( f;\delta \right) \\
&\leq &\left\{ 1+\frac{1}{\delta }\sqrt{\Omega _{2}}\right\} w\left(
f;\delta \right) \\
&=&\left\{ 1+\frac{1}{\delta }\frac{1}{\sqrt{n}}\sqrt{n\Omega _{2}}\right\}
w\left( f;\delta \right) \text{.}
\end{eqnarray*}%
If we put $\delta =\frac{1}{\sqrt{n}}$, we arrive at desired result.
\end{proof}

Functions $\varphi $ satisfying the following inequality are referred to as H%
\"{o}lder functions with exponent $\beta >0$ on $\left[ 0,\infty \right) $
if there is a positive constant $M$ such that%
\begin{equation*}
\left\vert \varphi \left( \zeta _{1}\right) -\varphi \left( \zeta
_{2}\right) \right\vert \leq M\left\vert \zeta _{1}-\zeta _{2}\right\vert
^{\beta }
\end{equation*}%
for all $\zeta _{1},\zeta _{2}\in \left[ 0,\infty \right) $.

\begin{theorem}
If $\varphi $ satisfies a H\"{o}lder condition with exponent $\beta >0$,
then the following estimate
\begin{equation*}
\left\vert K_{n}^{\mu }\left( \varphi ;x\right) -\varphi \left( x\right)
\right\vert \leq M\left[ \Omega _{2}\right] ^{\frac{\beta }{2}}
\end{equation*}%
holds.
\end{theorem}

\begin{proof}
Monotonicity property of $K_{n}^{\mu }$, combined with the Lemma 2, shows
that%
\begin{eqnarray*}
\left\vert K_{n}^{\mu }\left( \varphi ;x\right) -\varphi \left( x\right)
\right\vert  &=&\left\vert K_{n}^{\mu }\left( \varphi \left( \xi \right)
-\varphi \left( x\right) ;x\right) \right\vert  \\
&\leq &K_{n}^{\mu }\left( \left\vert \varphi \left( \xi \right) -\varphi
\left( x\right) \right\vert ;x\right)  \\
&\leq &MK_{n}^{\mu }\left( \left\vert \xi -x\right\vert ^{\beta };x\right)
\text{.}
\end{eqnarray*}%
Finally, from H\"{o}lder inequality we deduce the following expression
\begin{eqnarray*}
\frac{1}{Q\left( 1\right) e_{\mu }\left( nx\right) }\sum_{i=0}^{\infty }%
\frac{q_{i}\left( nx\right) }{\gamma _{\mu }\left( i\right) }\left\vert
\frac{i+2\mu \theta _{i}}{n}-x\right\vert ^{\beta } &=&\frac{1}{Q\left(
1\right) e_{\mu }\left( nx\right) } \\
&&\times \sum_{i=0}^{\infty }\left( \frac{q_{i}\left( nx\right) }{\gamma
_{\mu }\left( i\right) }\right) ^{\frac{2-\beta }{2}}\left( \frac{%
q_{i}\left( nx\right) }{\gamma _{\mu }\left( i\right) }\right) ^{\frac{\beta
}{2}}\left\vert \frac{i+2\mu \theta _{i}}{n}-x\right\vert ^{\beta } \\
&\leq &\frac{1}{Q\left( 1\right) e_{\mu }\left( nx\right) }\left[
\sum_{i=0}^{\infty }\frac{q_{i}\left( nx\right) }{\gamma _{\mu }\left(
i\right) }\right] ^{\frac{2-\beta }{2}} \\
&&\times \left[ \sum_{i=0}^{\infty }\frac{q_{i}\left( nx\right) }{\gamma
_{\mu }\left( i\right) }\left( \frac{i+2\mu \theta _{i}}{n}-x\right) ^{2}%
\right] ^{\frac{\beta }{2}} \\
&=&\left[ K_{n}^{\mu }\left( 1;x\right) \right] ^{\frac{2-\beta }{2}}\left[
\Omega _{2}\right] ^{\frac{\beta }{2}}\text{.}
\end{eqnarray*}%
So, the theorem is established.
\end{proof}

\begin{theorem}
Under the condition $\psi \in C\left[ 0,\infty \right) $, the relation
\begin{equation*}
\left\vert K_{n}^{\mu }\left( \psi ;x\right) -\psi \left( x\right)
\right\vert \leq \frac{3}{4}\left( 2+a+s^{2}\right) w_{2}\left( \psi
;s\right) +\frac{2s^{2}}{a}\left\Vert \psi \right\Vert \text{, }x\in \left[
0,a\right]
\end{equation*}%
holds, where $s=\sqrt[4]{\Omega _{2}}$ and $\left\Vert \psi \right\Vert
=\sup\limits_{x\in \left[ 0,\infty \right) }\left\vert \psi \left( x\right)
\right\vert $ and $w_{2}$ is second-order modulus of continuity.
\end{theorem}

\begin{proof}
Suppose that $a$ is a positive real number. By virtue of Lemma 2 and on the
basis of linearity property of $K_{n}^{\mu }$,
\begin{equation*}
K_{n}^{\mu }\left( \psi ;x\right) -\psi \left( x\right) =K_{n}^{\mu }\left(
\psi -\psi _{s};x\right) +K_{n}^{\mu }\left( \psi _{s};x\right) -\psi
_{s}\left( x\right) +\psi _{s}\left( x\right) -\psi \left( x\right)
\end{equation*}%
is verified. Taking into consideration the above equality we have, by virtue
of Lemma 2,
\begin{eqnarray*}
\left\vert K_{n}^{\mu }\left( \psi ;x\right) -\psi \left( x\right)
\right\vert  &\leq &K_{n}^{\mu }\left( \left\vert \psi -\psi _{s}\right\vert
;x\right) +\left\vert K_{n}^{\mu }\left( \psi _{s};x\right) -\psi _{s}\left(
x\right) \right\vert +\left\vert \psi _{s}\left( x\right) -\psi \left(
x\right) \right\vert  \\
&\leq &2\left\Vert \psi -\psi _{s}\right\Vert +\left\vert K_{n}^{\mu }\left(
\psi _{s};x\right) -\psi _{s}\left( x\right) \right\vert \text{.}
\end{eqnarray*}%
In accordance with the property of second-order Steklov function of the $%
\psi $ function \cite{17}, $\psi _{s}\in C^{2}\left[ 0,a\right] $ is
satisfied. Therefore, from the expansion of Taylor and inequality of
Cauchy-Schwarz, the right member of the above inequality%
\begin{equation}
\left\vert K_{n}^{\mu }\left( \psi _{s};x\right) -\psi _{s}\left( x\right)
\right\vert \leq \left\Vert \psi _{s}^{^{\prime }}\right\Vert \sqrt{\Omega
_{2}}+\frac{1}{2}\left\Vert \psi _{s}^{^{^{\prime \prime }}}\right\Vert
\Omega _{2}  \label{15}
\end{equation}%
can be written \cite{18}. As a consequence of Landau inequality we have the
following fact%
\begin{eqnarray*}
\left\Vert \psi _{s}^{^{\prime }}\right\Vert  &\leq &\frac{2}{a}\left\Vert
\psi _{s}\right\Vert +\frac{a}{2}\left\Vert \psi _{s}^{^{^{\prime \prime
}}}\right\Vert  \\
&\leq &\frac{2}{a}\left\Vert \psi \right\Vert +\frac{3a}{4}\frac{1}{s^{2}}%
w_{2}\left( f;s\right) \text{.}
\end{eqnarray*}%
Combining this formula with $\left( \ref{15}\right) $ gives us
\begin{eqnarray*}
\left\vert K_{n}^{\mu }\left( \psi ;x\right) -\psi \left( x\right)
\right\vert  &\leq &2\left\Vert \psi -\psi _{s}\right\Vert +\left\vert
K_{n}^{\mu }\left( \psi _{s};x\right) -\psi _{s}\left( x\right) \right\vert
\\
&\leq &\frac{3}{2}w_{2}\left( \psi ;s\right) +\left( \frac{2}{a}\left\Vert
\psi \right\Vert +\frac{3a}{4s^{2}}w_{2}\left( \psi ;s\right) \right) \sqrt{%
\Omega _{2}} \\
&&+\frac{3}{4s^{2}}w_{2}\left( \psi ;s\right) \Omega _{2}\text{.}
\end{eqnarray*}%
If we select $s=\sqrt[4]{\Omega _{2}}$, above inequality implies that%
\begin{equation*}
\left\vert K_{n}^{\mu }\left( \psi ;x\right) -\psi \left( x\right)
\right\vert \leq \frac{3}{4}\left( 2+a+s^{2}\right) w_{2}\left( \psi
;s\right) +\frac{2s^{2}}{a}\left\Vert \psi \right\Vert \text{.}
\end{equation*}%
This completes the proof.
\end{proof}

The following significant example is an explicit form of the $K_{n}^{\mu }$
operators.

\begin{example}
Polynomials having the following generating functions%
\begin{equation}
e^{a\omega ^{d+1}}e_{\mu }\left( x\omega \right) =\sum\limits_{i=0}^{\infty
}g_{i}^{\left( d+1\right) }\left( x,a,\mu \right) \frac{\omega ^{i}}{\gamma
_{\mu }\left( i\right) }  \label{16}
\end{equation}%
are called Gould-Hopper type polynomials \cite{15}. Gould-Hopper type
polynomials set $\left\{ g_{i}^{\left( d+1\right) }\left( x,a,\mu \right)
\right\} _{i=0}^{\infty }$ is a $d$ -orthogonal polynomial set \cite{15}.
From $\left( \ref{16}\right) $, it is clear that Gould-Hopper type
polynomials are the $\Lambda _{\mu }$ Appell polynomial set with%
\begin{equation*}
Q\left( \omega \right) =e^{a\omega ^{d+1}}\text{.}
\end{equation*}%
Under the assumption $a\geq 0$, $K_{n}^{\mu }$ operators which include
Gould-Hopper type polynomials are
\begin{equation*}
K_{n}^{\mu }\left( f;x\right) =\frac{e^{-a}}{e_{\mu }\left( nx\right) }%
\sum\limits_{i=0}^{\infty }\frac{g_{i}^{\left( d+1\right) }\left( nx,a,\mu
\right) }{\gamma _{\mu }\left( i\right) }f\left( \frac{i+2\mu \theta _{i}}{n}%
\right) \text{.}
\end{equation*}
\end{example}

\end{document}